\documentclass[oneside,english]{amsart}
\usepackage{ae,aecompl}
\usepackage[T1]{fontenc}
\usepackage[latin9]{inputenc}
\usepackage{amsthm}
\usepackage{amssymb}
\usepackage{esint}

\makeatletter
\numberwithin{equation}{section} 
\numberwithin{figure}{section} 
\theoremstyle{plain}
\newtheorem{thm}{Theorem}
  \theoremstyle{plain}
  \newtheorem{prop}[thm]{Proposition}
  \theoremstyle{plain}
  \newtheorem{cor}[thm]{Corollary}

\makeatother

\usepackage{babel}

\begin{document}
\subjclass[2000]{Primary 37A40, 37A35, 60G51; secondary 37A15, 37A50}

\title{Poisson-Pinsker factor and infinite measure preserving group actions}

\author{Emmanuel Roy}

\curraddr{Laboratoire Analyse Géométrie et Applications, UMR 7539, Université
Paris 13, 99 avenue J.B. Clément, F-93430 Villetaneuse, France}

\email{roy@math.univ-paris13.fr}
\begin{abstract}
We solve the question of the existence of a Poisson-Pinsker factor
for conservative ergodic infinite measure preserving action of a countable
amenable group by proving the following dichotomy: either it has totally
positive Poisson entropy (and is of zero type), or it possesses a
Poisson-Pinsker factor. If $G$ is abelian and the entropy positive,
the spectrum is absolutely continuous (Lebesgue countable if $G=\mathbb{Z}$)
on the whole $L^{2}$-space in the first case and in the orthocomplement
of the $L^{2}$-space of the Poisson-Pinsker factor in the second.
\end{abstract}

\keywords{Poisson suspensions, infinite ergodic theory, joinings}

\thanks{This paper has been written during the MSRI program {}``Ergodic
Theory and Additive Combinatorics'' in Berkeley. The author is very
grateful to this institution and to the organizers of this program
for funding his research during this semester.}

\maketitle

\section{Introduction}

Poisson entropy for a $\sigma$-finite measure preserving $\mathbb{Z}$-action
was introduced in \cite{Roy05these} as the Kolmogorov entropy of
its Poisson suspension (see Section \ref{sec:Background-on-Poisson}
for the definition of a Poisson suspension), it is a non-trivial invariant
which coincides with Kolmogorov entropy in the finite measure case
and retains most of its basic features. The definition readily extends
to more general group actions as soon as it is well defined in the
finite measure preserving case. For conservative $\mathbb{Z}$-action,
there exist other definitions of entropy due to Krengel and Parry
(see \cite{Kren67Entropy} and \cite{Par69Entropy}) and it is proved
in \cite{Jan08entropy} that these two entropies equal Poisson entropy
for a large family of systems called {}``quasi-finite'', however
it has been proved recently in \cite{delarueJan09KredifPoi} that
Krengel entropy can be different from Poisson and Parry entropies.
The question of the existence of a Pinsker factor for the Krengel
entropy (a \emph{Krengel-Pinsker factor}) was already addressed (and
unsolved) in the original Krengel's paper. The first partial answer
was due to Aaronson and Park in \cite{AarPark07ent} where they proved
the existence of a Krengel-Pinsker factor for LLB systems (a subfamily
of quasi-finite systems). In \cite{Jan08entropy}, it is proved that
as soon as there exists a ($\sigma$-finite) factor of zero Poisson
entropy, there exists a maximal factor with zero Poisson entropy (a
\emph{Poisson-Pinsker factor}).

The method used to obtain the existence of Poisson-Pinsker factor
was very particular to $\mathbb{Z}$-actions and was only partially
satisfactory since it was not possible to prove that a system generated
by a square integrable function with singular spectral measure had
zero Poisson entropy nor a conservative system with finite spectral
multiplicity, as it is well known in the finite measure case. In this
short paper, we will get, with completely different methods, a much
more satisfactory picture of the situation (extended to countable
amenable group actions) by showing, more generally, that the Pinsker
factor of a Poisson suspension, if not trivial, is always (isomorphic
to) a Poisson suspension.

These results strongly rely on the structure of joinings of Poisson
suspensions (see \cite{Lem05ELF}, \cite{Roy05these} and \cite{Roy07Infinite})
on the one hand, and entropy results for countable amenable group
actions on the other (see \cite{GlaThouWei00EntropywoPast}, \cite{DooGol02specamenable}
and \cite{DanRud07Condent}), in particular the relative disjointness
results of Thouvenot \cite{Thouv75Pinsker} for $\mathbb{Z}$-actions,
generalized in \cite{GlaThouWei00EntropywoPast} for countable amenable
group actions.

\section{\label{sec:Background-on-Poisson}Background on Poisson suspensions}

We recall a few facts about Poisson suspensions that can be found
in \cite{Roy07Infinite}. Note that in this section, we mention structural
features and results that were once proved in the case of $\mathbb{Z}$-actions,
however, it can be seen that the extension to more general actions
is immediate. In the following $G$ is a countable amenable group
with identity element $e$.

\subsection{Definition of Poisson suspensions}

Let $\left(X,\mathcal{A},\mu\right)$ a $\sigma$-finite Lebesgue
space with an infinite continuous measure $\mu$. Let us define the
probability space $\left(X^{*},\mathcal{A}^{*},\mu^{*}\right)$ where:
\begin{itemize}
\item $X^{*}$ stands for the space of measures of the form $\nu=\sum_{n\in\mathbb{N}}\delta_{x_{i}}$
where $x_{i}\in X$.
\item $\mathcal{A}^{*}$ is the smallest $\sigma$-algebra such that, for
any $A\in\mathcal{A}$, the map $N_{A}:\;\nu\mapsto\nu\left(A\right)$
from $\left(X,\mathcal{A}\right)$ to $\left(\overline{\mathbb{N}},\mathcal{P}\left(\overline{\mathbb{N}}\right)\right)$
is measurable.
\item $\mu^{*}$ is the only probability measure on $\left(X^{*},\mathcal{A}^{*}\right)$
such that, for any integer $k$, and any collection of disjoint sets
$\left\{ A_{1},\dots,A_{k}\right\} $ of finite measure, the random
variables $N_{A_{1}},\dots,N_{A_{k}}$\emph{ }are independent and
follow a Poisson distribution with parameter $\mu\left(A_{1}\right),\dots,\mu\left(A_{k}\right)$
respectively.
\end{itemize}
From now on, $\mathcal{A}^{*}$ is assumed complete with respect to
$\mu^{*}$, the probability space $\left(X^{*},\mathcal{A}^{*},\mu^{*}\right)$
is a Lebesgue space called the \emph{Poisson measure} over $\left(X,\mathcal{A},\mu\right)$.

Let $G$ be a countable group acting on $\left(X,\mathcal{A},\mu\right)$
by measure preserving automorphisms $T^{g}$, $g\in G$. If $T$ is
a measure preserving automorphism of $\left(X,\mathcal{A},\mu\right)$,
then $T_{*}$ defined on $\left(X^{*},\mathcal{A}^{*}\right)$ by
$T_{*}\left(\nu\right)\left(A\right)=\nu\left(T^{-1}A\right)$, $A\in\mathcal{A}$,
is a measure preserving automorphism of $\left(X^{*},\mathcal{A}^{*},\mu^{*}\right)$.
$\left(X^{*},\mathcal{A}^{*},\mu^{*},T_{*}^{g}\right)$ is called
the \emph{Poisson suspension} over the \emph{base} $\left(X,\mathcal{A},\mu,T^{g}\right)$.

\subsection{Poisson entropy}

The \emph{Poisson entropy} of a system $\left(X,\mathcal{A},\mu,T^{g}\right)$
is defined as the usual (Kolmogorov) entropy of its Poisson suspension
$\left(X^{*},\mathcal{A}^{*},\mu^{*},T_{*}^{g}\right)$. $\left(X,\mathcal{A},\mu,T^{g}\right)$
will be said to have \emph{totally positive Poisson entropy} if, for
any invariant set $K\subset X$ of positive measure the Poisson entropy
of the restricted system $\left(K,\mathcal{A}_{\mid K},\mu_{\mid K},T^{g}\right)$
is positive on ANY factor. Note that we reserve the terminology \emph{factor}
to invariant sub-$\sigma$-algebras on which the measure is still
$\sigma$-finite (the trivial $\sigma$-algebra $\left\{ X,\emptyset\right\} \subset\mathcal{A}$
is NOT a factor of $\left(X,\mathcal{A},\mu,T^{g}\right)$), therefore,
\emph{complete positive entropy} (cpe) is a different notion which
is not adapted to the infinite measure context. Of course a probability
measure preserving system never has totally positive entropy since
the trivial algebra is a factor.

\subsection{\label{sub:Fock-space-structure}Fock space structure}

It is classical to see $L^{2}\left(\mathcal{A}^{*}\right)$ as the
Fock space over $L^{2}\left(\mathcal{A}\right)$, that is:

\[
L^{2}\left(\mathcal{A}^{*}\right)\simeq\mathbb{C}\oplus L^{2}\left(\mathcal{A}\right)\oplus L^{2}\left(\mathcal{A}\right)^{\odot2}\otimes\cdots\otimes L^{2}\left(\mathcal{A}\right)^{\odot n}\otimes\cdots\]
where $L^{2}\left(\mathcal{A}\right)^{\odot n}$ stands for the symmetric
tensor product of $L^{2}\left(\mathcal{A}\right)$ and is called the
\emph{$n$-th chaos}. Within $L^{2}\left(\mathcal{A}^{*}\right)$,
the subspace correcponding to the first chaos is noted $\mathfrak{C}$
and is the closure of the linear span of vectors $N\left(A\right)-\mu\left(A\right)$,
$A\in\mathcal{A}$, $\mu\left(A\right)<\infty$. A linear operator
$U$ on $L^{2}\left(\mathcal{A}\right)$ of norm less than $1$ induces
an operator $\widetilde{U}$ on $L^{2}\left(\mathcal{A}^{*}\right)$
when defined on $L^{2}\left(\mathcal{A}\right)^{\odot n}$ by $\widetilde{U}\left(f\otimes\cdots\otimes f\right)=Uf\otimes\cdots\otimes Uf$.
We have in particular $\widetilde{U_{T^{g}}}=U_{T_{*}^{g}}$.

There is a distinguished family of vectors, linearly dense in $L^{2}\left(\mathcal{A}^{*}\right)$;
namely, for $f\in L^{2}\left(\mathcal{A}\right)$ which are finite
linear combination of indicator functions, \[
\mathcal{E}_{f}:=1\oplus f\oplus\frac{1}{\sqrt{2!}}f\otimes f\oplus\dots\oplus\frac{1}{\sqrt{n!}}f\otimes\cdots\otimes f,\]
 which corresponds, through the above identification, to:\[
\mathcal{E}_{f}\left(\nu\right)=\exp\left(-\int_{X}fd\mu\right){\displaystyle \prod_{x,\nu\left(x\right)=1}}\left(1+f\left(x\right)\right),\quad\nu\in X^{*}.\]

In case $G$ is abelian, if $\sigma$ is the maximal spectral type
of $U_{T^{g}}$ on $L^{2}\left(\mathcal{A}\right)$ then the reduced
maximal spectral type of $U_{T_{*}^{g}}$ is $\sum_{k=1}^{\infty}\frac{1}{k!}\sigma^{*k}$
since the maximal spectral type on the $k$-th chaos is $\sigma^{*k}$.

\subsection{Poissonian factors}

A \emph{Poissonian factor} is a sub-$\sigma$-algebra of $\mathcal{A}^{*}$
of the form $\left(\mathcal{B}_{\mid K}\right)^{*}:=\sigma\left\{ N_{A},\; A\in\mathcal{B}_{\mid K}\right\} $
where $K\subset X$ is a $T^{g}$-invariant measurable set of positive
$\mu$-measure and $\mathcal{B}_{\mid K}$ is a factor of the restricted
system $\left(K,\mathcal{A}_{\mid K},\mu_{\mid K},T^{g}\right)$.
In terms of systems, the factor $\left(\mathcal{B}_{\mid K}\right)^{*}$
corresponds to the Poisson suspension $\left(K^{*},\left(\mathcal{B}_{\mid K}\right)^{*},\left(\mu_{\mid K}\right)^{*},T_{*}^{g}\right)$.
The trivial factor $\left\{ X^{*},\emptyset\right\} \subset\mathcal{A}^{*}$
is also considered as Poissonian.

\subsection{Infinite divisibility and Poissonian joinings}

Addition is well defined on $\left(X^{*},\mathcal{A}^{*}\right)$
as the usual sum of measures and so is convolution of probability
measures on $\left(X^{*},\mathcal{A}^{*}\right)$. A probability measure
$p$ on $\left(X^{*},\mathcal{A}^{*}\right)$ such that, for any integer
$k$, there exists a probability measure $p_{k}$ satisfying $p=p_{k}*\cdots*p_{k}$
($k$ terms) is said to be \emph{infinitely divisible}. It is well
known that $\mu^{*}$ is infinitely divisible as $\mu^{*}=\left(\frac{1}{k}\mu\right)^{*}*\cdots*\left(\frac{1}{k}\mu\right)^{*}$.

Addition on the product space $\left(X^{*}\times X^{*},\mathcal{A}^{*}\otimes\mathcal{A}^{*}\right)$
is defined coordinate wise and so is convolution and infinite divisibility.
A self-joining of a Poisson suspension $\left(X^{*},\mathcal{A}^{*},\mu^{*},T_{*}^{g}\right)$
is said to be a \emph{Poissonian self-joining} if its distribution
on $\left(X^{*}\times X^{*},\mathcal{A}^{*}\otimes\mathcal{A}^{*}\right)$
is infinitely divisible. In \cite{Roy07Infinite}, it is proved:
\begin{prop}
A self-joining, determined by a Markov operator $\Psi$ on $L^{2}\left(\mathcal{A}^{*}\right)$,
is a Poissonian self-joining if and only if there exists a sub-Markov
operator $\Phi$ on $L^{2}\left(\mathcal{A}\right)$ (i.e. a positive
operator such that $\Phi\left(1\right)\leq1$ and $\Phi^{*}\left(1\right)\leq1$
) such that $\widetilde{\Phi}=\Psi$ (see \cite{Lem05ELF} and \cite{Roy07Infinite}).
\end{prop}
Poissonian self-joinings were originally introduced independently
through their Markov operator characterization in \cite{Lem05ELF},
and their infinitely divisible one in \cite{Roy05these}.

The next proposition is found in \cite{Roy07Infinite}, however, we
give a proof as the original one is too sketchy:
\begin{prop}
\label{pro:Poissonfactor} If $\mathcal{F}$ is a non-trivial factor
of $\left(X^{*},\mathcal{A}^{*},\mu^{*},T_{*}^{g}\right)$ and if
the relatively independent product of $\left(X^{*},\mathcal{A}^{*},\mu^{*},T_{*}^{g}\right)$
over $\mathcal{F}$ is a Poissonian joining, then $\mathcal{F}$ is
a Poissonian factor, that is there exists a $T^{g}$-invariant subset
$K\subset X$ and a factor $\mathcal{B}_{\mid K}\subset\mathcal{A}_{\mid K}$
such that $\left(\mathcal{B}_{\mid K}\right)^{*}=\mathcal{F}$.\end{prop}
\begin{proof}
Let $\pi_{\mathcal{F}}$ the projection on $L^{2}\left(\mathcal{F}\right)$
that corresponds to the relatively independent joining over $\mathcal{F}$.
Since this joining is Poissonian, there exists a sub-Markov $\Phi$
operator on $L^{2}\left(\mathcal{A}\right)$ such that $\widetilde{\Phi}=\pi_{\mathcal{F}}$.
Observe that $\Phi$ is also an orthogonal projection since it is
induced by $\pi_{\mathcal{F}}$ restricted to the first chaos, we
therefore have:\[
\int_{X}\Phi\left(1_{A}\right)d\mu=\int_{X}1_{A}\Phi\left(1\right)d\mu.\]

But we also have\[
\int_{X}\Phi\left(1_{A}\right)d\mu=\int_{X}\Phi\left(\Phi\left(1_{A}\right)\right)d\mu=\int_{X}\Phi\left(1_{A}\right)\Phi\left(1\right)d\mu,\]

and this implies\[
\int_{X}\Phi\left(1_{A}\right)\left(1-\Phi\left(1\right)\right)d\mu=0.\]

Therefore, as $\Phi$ is positive and $\Phi\left(1\right)\leq1$,
$\mu$-a.e.\[
\Phi\left(1_{A}\right)\left(1-\Phi\left(1\right)\right)=0.\]

Replacing $A$ by a sequence $A_{n}$ of finite measure sets increasing
to $X$, we get:\[
\Phi\left(1\right)\left(1-\Phi\left(1\right)\right)=0\]
 and\[
\Phi\left(1\right)=\left(\Phi\left(1\right)\right)^{2}.\]

So $\Phi\left(1\right)$ is the indicator function of a set $K$ and
if we consider the restricted system $\left(K,\mathcal{A}_{\mid K},\mu_{\mid K},T^{g}\right)$,
$\Phi$ is a Markov operator and an orthogonal projection of $L^{2}\left(\mathcal{A}_{\mid K}\right)$,
thus a conditional expectation on a factor $\mathcal{B}_{\mid K}\subset\mathcal{A}_{\mid K}$.
Now if $f\in L^{2}\left(\mathcal{A}\right)$, $\mathcal{E}_{\Phi f}$
is $\left(\mathcal{B}_{\mid K}\right)^{*}$-measurable and, for any
$g\in L^{2}\left(\mathcal{B}_{\mid K}\right)$

\[
\left\langle \pi_{\mathcal{F}}\mathcal{E}_{f},\mathcal{E}_{g}\right\rangle _{L^{2}\left(\mathcal{A}^{*}\right)}=\left\langle \mathcal{E}_{\Phi f},\mathcal{E}_{g}\right\rangle _{L^{2}\left(\mathcal{A}^{*}\right)}=\exp\left\langle \Phi f,g\right\rangle _{L^{2}\left(\mathcal{B}_{\mid K}\right)}\]

\[
=\exp\left\langle f,g\right\rangle _{L^{2}\left(\mathcal{B}_{\mid K}\right)}=\exp\left\langle f,g\right\rangle _{L^{2}\left(\mathcal{A}\right)}=\left\langle \mathcal{E}_{f},\mathcal{E}_{g}\right\rangle _{L^{2}\left(\mathcal{A}^{*}\right)}=\left\langle \pi_{\mathcal{B}_{\mid K}^{*}}\mathcal{E}_{f},\mathcal{E}_{g}\right\rangle _{L^{2}\left(\mathcal{A}^{*}\right)}\]

with some slight abuses in notation. Therefore $\mathcal{F}=\left(\mathcal{B}_{\mid K}\right)^{*}$.
\end{proof}

\section{The main results}
\begin{prop}
\label{prop:Main Result}Let $\left(X^{*},\mathcal{A}^{*},\mu^{*},T_{*}^{g}\right)$
be the Poisson suspension of the infinite measure preserving system
$\left(X,\mathcal{A},\mu,G\right)$. Then $\Pi$, the Pinsker factor
of the system, is a Poissonian factor.\end{prop}
\begin{proof}
Take $\alpha,\beta>0$ and consider the direct product \[
\left(X^{*}\times X^{*},\mathcal{A}^{*}\otimes\mathcal{A}^{*},\left(\alpha\mu\right)^{*}\otimes\left(\beta\mu\right)^{*},T_{*}^{g}\times T_{*}^{g}\right)\]
 of Poisson suspensions. Thanks to the classical formula $\left(\alpha\mu\right)^{*}*\left(\beta\mu\right)^{*}=\left(\alpha+\beta\right)\mu^{*}$,
the map $\varphi:\,\left(\nu_{1},\nu_{2}\right)\mapsto\nu_{1}+\nu_{2}$
induces a factor map from \[
\left(X^{*}\times,X^{*}\mathcal{A}^{*}\otimes\mathcal{A}^{*},\left(\alpha\mu\right)^{*}\otimes\left(\beta\mu\right)^{*},T_{*}^{g}\times T_{*}^{g}\right)\]
 to \[
\left(X^{*},\mathcal{A}^{*},\left(\left(\alpha+\beta\right)\mu\right)^{*},T_{*}^{g}\right).\]
Denote by $\Pi_{\alpha}$ and $\Pi_{\beta}$ and $\Pi_{\alpha+\beta}$
the Pinsker algebra of the systems $\left(X^{*},\mathcal{A}^{*},\left(\alpha\mu\right)^{*},T_{*}^{g}\right)$,
$\left(X^{*},\mathcal{A}^{*},\left(\beta\mu\right)^{*},T_{*}^{g}\right)$
and $\left(X^{*},\mathcal{A}^{*},\left(\left(\alpha+\beta\right)\mu\right)^{*},T_{*}^{g}\right)$
respectively and set $\mathcal{B}:=\varphi^{-1}\mathcal{A}^{*}$ and
$\widetilde{\Pi_{\alpha+\beta}}:=\varphi^{-1}\Pi_{\alpha+\beta}$.
Thanks to a classical result (generalized to countable amenable group
actions in Theorem 4 in \cite{GlaThouWei00EntropywoPast}), the Pinsker
algebra of the product $\left(X^{*}\times X^{*},\mathcal{A}^{*}\otimes\mathcal{A}^{*},\left(\alpha\mu\right)^{*}\otimes\left(\beta\mu\right)^{*},T_{*}^{g}\times T_{*}^{g}\right)$
is $\Pi_{\alpha}\otimes\Pi_{\beta}$. The extension $\mathcal{B}\to\widetilde{\Pi_{\alpha+\beta}}$
is a cpe extension and $\Pi_{\alpha}\otimes\Pi_{\beta}\to\widetilde{\Pi_{\alpha+\beta}}$
is a zero entropy extension. Therefore, by Lemma 3 in \cite{Thouv75Pinsker}
(once again generalized to countable amenable group actions in Theorem
1 in \cite{GlaThouWei00EntropywoPast}), they are relatively disjoint
over $\widetilde{\Pi_{\alpha+\beta}}$. As a consequence, $L^{2}\left(\mathcal{B}\right)\ominus L^{2}\left(\widetilde{\Pi_{\alpha+\beta}}\right)$
and $L^{2}\left(\Pi_{\alpha}\otimes\Pi_{\beta}\right)\ominus L^{2}\left(\widetilde{\Pi_{\alpha+\beta}}\right)$
are orthogonal in $L^{2}\left(\mathcal{A}^{*}\otimes\mathcal{A}^{*}\right)$.
Indeed if $f\in L^{2}\left(\mathcal{B}\right)\ominus L^{2}\left(\widetilde{\Pi_{\alpha+\beta}}\right)$
and $g\in L^{2}\left(\Pi_{\alpha}\otimes\Pi_{\beta}\right)\ominus L^{2}\left(\widetilde{\Pi_{\alpha+\beta}}\right)$,
we have:\begin{eqnarray*}
 &  & \mathbb{E}\left[fg\right]\\
 &  & =\mathbb{E}\left[\mathbb{E}\left[fg\mid\widetilde{\Pi_{\alpha+\beta}}\right]\right]\\
 &  & =\mathbb{E}\left[\mathbb{E}\left[f\mid\widetilde{\Pi_{\alpha+\beta}}\right]\mathbb{E}\left[g\mid\widetilde{\Pi_{\alpha+\beta}}\right]\right]\\
 &  & =0\end{eqnarray*}

We can therefore decompose $L^{2}\left(\mathcal{A}^{*}\otimes\mathcal{A}^{*}\right)$
into the following orthogonal sum:

\begin{eqnarray*}
 &  & L^{2}\left(\mathcal{A}^{*}\otimes\mathcal{A}^{*}\right)\\
 &  & =L^{2}\left(\widetilde{\Pi_{\alpha+\beta}}\right)\oplus\left(L^{2}\left(\mathcal{B}\right)\ominus L^{2}\left(\widetilde{\Pi_{\alpha+\beta}}\right)\right)\oplus\left(L^{2}\left(\Pi_{\alpha}\otimes\Pi_{\beta}\right)\ominus L^{2}\left(\widetilde{\Pi_{\alpha+\beta}}\right)\right)\oplus H\end{eqnarray*}

where $H$ is the orthocomplement of everything else. Now write $f\in L^{2}\left(\mathcal{A}^{*}\otimes\mathcal{A}^{*}\right)$
as $f=f_{1}+f_{2}+f_{3}+f_{4}$ according to the decomposition, we
have:

\[
\mathbb{E}\left[\mathbb{E}\left[f\mid\mathcal{B}\right]\mid\Pi_{\alpha}\otimes\Pi_{\beta}\right]=\mathbb{E}\left[\left(f_{1}+f_{2}\right)\mid\Pi_{\alpha}\otimes\Pi_{\beta}\right]=f_{1}\]

thus $\mathbb{E}\left[f\mid\widetilde{\Pi_{\alpha+\beta}}\right]=\mathbb{E}\left[\mathbb{E}\left[f\mid\mathcal{B}\right]\mid\Pi_{\alpha}\otimes\Pi_{\beta}\right]$.

Now form the relatively independent joining of \[
\left(X^{*}\times,X^{*}\mathcal{A}^{*}\otimes\mathcal{A}^{*},\left(\alpha\mu\right)^{*}\otimes\left(\beta\mu\right)^{*},T_{*}^{g}\times T_{*}^{g}\right)\]
 over $\Pi_{\alpha}\otimes\Pi_{\beta}$ and remark that it is just
the direct product of the relatively independent joinings of $\left(X^{*},\mathcal{A}^{*},\left(\alpha\mu\right)^{*},T_{*}^{g}\right)$
and $\left(X^{*},\mathcal{A}^{*},\left(\beta\mu\right)^{*},T_{*}^{g}\right)$
over their respective Pinsker factors $\Pi_{\alpha}$ and $\Pi_{\beta}$.
Let's now compute the distribution of the self-joining of $\left(X^{*},\mathcal{A}^{*},\mu^{*},T_{*}^{g}\right)$
induced by the preceding joining, through the factor map $\varphi\times\varphi$.
Take $A$ and $B$ in $\mathcal{B}$ and compute:

\begin{eqnarray*}
 &  & \mathbb{E}\left[\mathbb{E}\left[1_{A}\mid\Pi_{\alpha}\otimes\Pi_{\beta}\right]\mathbb{E}\left[1_{B}\mid\Pi_{\alpha}\otimes\Pi_{\beta}\right]\right]\\
 &  & =\mathbb{E}\left[\mathbb{E}\left[\mathbb{E}\left[1_{A}\mid\mathcal{B}\right]\mid\Pi_{\alpha}\otimes\Pi_{\beta}\right]\mathbb{E}\left[\mathbb{E}\left[1_{B}\mid\mathcal{B}\right]\mid\Pi_{\alpha}\otimes\Pi_{\beta}\right]\right]\\
 &  & =\mathbb{E}\left[\mathbb{E}\left[1_{A}\mid\widetilde{\Pi_{\alpha+\beta}}\right]\mathbb{E}\left[1_{B}\mid\widetilde{\Pi_{\alpha+\beta}}\right]\right]\end{eqnarray*}

Therefore, the joining induced is nothing else than the relatively
independent joining over its Pinsker factor. We have just proved that
the image measure of \[
\left(\left(\alpha\mu\right)^{*}\otimes_{\Pi_{\alpha}}\left(\alpha\mu\right)^{*}\right)\otimes\left(\left(\beta\mu\right)^{*}\otimes_{\Pi_{\beta}}\left(\beta\mu\right)^{*}\right)\]
by the sum application $\varphi\times\varphi$ is the measure $\mu^{*}\otimes_{\Pi_{\alpha+\beta}}\mu^{*}$
, that is, we have proved the following formula:

\[
\left(\left(\alpha\mu\right)^{*}\otimes_{\Pi_{\alpha}}\left(\alpha\mu\right)^{*}\right)*\left(\left(\beta\mu\right)^{*}\otimes_{\Pi_{\beta}}\left(\beta\mu\right)^{*}\right)=\mu^{*}\otimes_{\Pi_{\alpha+\beta}}\mu^{*}\]

and we can deduce, for any integer $k$:\[
\left(\left(\frac{1}{k}\mu\right)^{*}\otimes_{\Pi_{\frac{1}{k}}}\left(\frac{1}{k}\mu\right)^{*}\right)^{*k}=\mu^{*}\otimes_{\Pi_{1}}\mu^{*}.\]

This means that the distribution of this relatively independent joining
is infinitely divisible, i.e. it is a Poissonian joining. But according
to Proposition \ref{pro:Poissonfactor}, $\Pi_{1}=\Pi$ is a Poissonian
factor.\end{proof}
\begin{prop}
Let $\left(X,\mathcal{A},\mu,T^{g}\right)$ be a dynamical system.
There exists a (possibly trivial) partition into $T^{g}$-invariant
sets $A$ and $A^{c}$ such that:
\begin{enumerate}
\item $\left(A,\mathcal{A}_{\mid A},\mu_{\mid A},T^{g}\right)$ possesses
a Poisson-Pinsker factor.
\item for any $T^{g}$-invariant set $B\subset A^{c}$ of positive measure,
$\left(B,\mathcal{A}_{\mid B},\mu_{\mid B},T^{g}\right)$ has totally
positive Poisson entropy.
\end{enumerate}
\end{prop}
\begin{proof}
If the Pinsker factor of $\left(X^{*},\mathcal{A}^{*},\mu^{*},T_{*}^{g}\right)$
is trivial, the system has complete positive entropy, therefore, for
any $T^{g}$-invariant set $B\subset X$ of positive measure, $\left(B,\mathcal{A}_{\mid B},\mu_{\mid B},T^{g}\right)$
has totally positive Poisson entropy since any factor $\mathcal{B}_{\mid B}$
corresponds to the non-trivial Poissonian factor $\left(\mathcal{B}_{\mid B}\right)^{*}$
on which the Poisson suspension has positive entropy. If the Pinsker
factor $\Pi$ of $\left(X^{*},\mathcal{A}^{*},\mu^{*},T_{*}^{g}\right)$
is not trivial, from Proposition \ref{prop:Main Result}, there exists
a $T^{g}$-invariant set $A\subset X$ of positive measure and a factor,
say $\mathcal{P}_{\mid A}$, of the restricted system $\left(A,\mathcal{A}_{\mid A},\mu_{\mid A},T^{g}\right)$
such that $\left(\mathcal{P}_{\mid A}\right)^{*}=\Pi$. $\mathcal{P}_{\mid A}$
is clearly the Poisson-Pinsker factor of the system $\left(A,\mathcal{A}_{\mid A},\mu_{\mid A},T^{g}\right)$.
Indeed, if $\mathcal{C}_{\mid A}\subset\mathcal{A}_{\mid A}$ is a
factor with zero Poisson entropy, $\left(\mathcal{C}_{\mid A}\right)^{*}\subset\left(\mathcal{P}_{\mid A}\right)^{*}$
as the latter is the Pinsker factor of the suspension, and this implies
$\mathcal{C}_{\mid A}\subset\mathcal{P}_{\mid A}$ (the fact that
a factor $\mathcal{R}$ is the Poisson-Pinsker factor if the associated
Poissonian factor $\mathcal{R}^{*}$ is the Pinsker factor of the
suspension was already observed in \cite{Jan08entropy}).

If $A^{c}$ has positive measure, the Poisson suspension $\left(X^{*},\mathcal{A}^{*},\mu^{*},T_{*}^{g}\right)$
splits into the direct product \[
\left(A^{*}\times\left(A^{c}\right)^{*},\left(\mathcal{A}_{\mid A}\right)^{*}\otimes\left(\mathcal{A}_{\mid A^{c}}\right)^{*},\left(\mu_{\mid A}\right)^{*}\otimes\left(\mu_{\mid A^{c}}\right)^{*},T_{*}^{g}\times T_{*}^{g}\right)\]

which implies that the Pinsker factor $\Pi$ also splits accordingly.
But as $\left(\mathcal{P}_{\mid A}\right)^{*}=\Pi$, the Pinsker factor
is concentrated in the first side of the product, that is, $\left(\left(A^{c}\right)^{*},\left(\mathcal{A}_{\mid A^{c}}\right)^{*},\left(\mu_{\mid A^{c}}\right)^{*},T_{*}^{g}\right)$
has complete positive entropy and we conclude as in the first part
of the proof.
\end{proof}
In the ergodic case, the result takes the following more pleasant
form:
\begin{thm}
Let $\left(X,\mathcal{A},\mu,T^{g}\right)$ be an ergodic infinite
measure preserving system. Either it has totally positive Poisson
entropy, or it possesses a Poisson-Pinsker factor.
\end{thm}
As in the $\mathbb{Z}$-action case, we can observe the behaviour
of Poisson entropy with respect to joinings.
\begin{prop}
Zero Poisson entropy is stable under taking joinings. Totally positive
entropy systems are strongly disjoint from systems possessing a Poisson-Pinsker
factor.\end{prop}
\begin{proof}
The first statement is obvious and the proof of the second is identical
to the $\mathbb{Z}$-action case which can be found in \cite{Jan08entropy}.
\end{proof}

\section{Spectral properties}

We first recall that a system $\left(X,\mathcal{A},\mu,T^{g}\right)$
is of \emph{zero type} if for any measurable sets $A$ and $B$ in
$\mathcal{A}$ of finite measure, $\mu\left(A\cap T^{g}B\right)$
tends to zero as $g$ tends to infinity.
\begin{prop}
If $\left(X,\mathcal{A},\mu,T^{g}\right)$ has totally positive entropy,
then it is of zero type.\end{prop}
\begin{proof}
From Proposition \ref{prop:Main Result} $\left(X^{*},\mathcal{A}^{*},\mu^{*},T_{*}^{g}\right)$
has complete positive entropy and is therefore mixing. But, thanks
to the classical isometry formula \[
\mu\left(A\cap B\right)=\mathbb{E}_{\mu^{*}}\left[\left(N_{A}-\mu\left(A\right)\right)\left(N_{B}-\mu\left(B\right)\right)\right]\]

we have

\begin{eqnarray*}
 &  & \mu\left(A\cap T^{g}B\right)\\
 &  & =\mathbb{E}_{\mu^{*}}\left[\left(N_{A}-\mu\left(A\right)\right)\left(N_{T^{g}B}-\mu\left(T^{g}B\right)\right)\right]\\
 &  & =\mathbb{E}_{\mu^{*}}\left[\left(N_{A}-\mu\left(A\right)\right)\left(N_{B}-\mu\left(B\right)\right)\circ T_{*}^{g}\right]\end{eqnarray*}

which goes to zero as $g$ tends to infinity.\end{proof}
\begin{prop}
\label{pro:Spectrum}If $G$ is abelian and $\left(X^{*},\mathcal{A}^{*},\mu^{*},T_{*}^{g}\right)$
has positive entropy, then it has absolutely continuous spectrum on
$\mathfrak{C}\cap\left(L^{2}\left(\Pi\right)^{\perp}\right)$.\end{prop}
\begin{proof}
Since, thanks to Proposition \ref{prop:Main Result}, $\Pi$ has the
structure of a Poissonian factor, its associated $L^{2}$-space is
a Fock space compatible with the underlying one, in particular $\mathfrak{C}=\left(L^{2}\left(\Pi\right)^{\perp}\cap\mathfrak{C}\right)\begin{array}[b]{c}
\perp\\
\oplus\end{array}\left(L^{2}\left(\Pi\right)\cap\mathfrak{C}\right)$. Since, by assumption $L^{2}\left(\Pi\right)\neq L^{2}\left(\mathcal{A}^{*}\right)$,
$\mathfrak{C}\not\subset L^{2}\left(\Pi\right)$ as $\sigma\left(\mathfrak{C}\right)=\mathcal{A}^{*}$.
Therefore $\left(L^{2}\left(\Pi\right)^{\perp}\cap\mathfrak{C}\right)$
is not empty and we get the result, as the maximal spectral type on
$L^{2}\left(\Pi\right)^{\perp}$ is Lebesgue by a Theorem proved in
\cite{DooGol02specamenable} and independently by Thouvenot (unpublished).
\end{proof}
As a direct application, we get:
\begin{prop}
\label{pro:SpectrumInfinitemeasure}Assume $G$ is abelian. If $\left(X,\mathcal{A},\mu,T^{g}\right)$
has totally positive Poisson entropy, its maximal spectral type is
absolutely continuous.

If $\left(X,\mathcal{A},\mu,T^{g}\right)$ has positive Poisson entropy
and possesses a Poisson-Pinsker factor, then the maximal spectral
type on the orthocomplement of the Poisson-Pinsker factor is absolutely
continuous.\end{prop}
\begin{proof}
This follows directly from Proposition \ref{pro:Spectrum} and from
the unitary isomorphism between the first chaos of $L^{2}\left(\mathcal{A}^{*}\right)$
and $L^{2}\left(\mathcal{A}\right)$.
\end{proof}
We therefore can deduce the following corollary, well known in the
finite measure case, the proof being almost the same:
\begin{cor}
Assume $\left(X,\mathcal{A},\mu,T^{g}\right)$ is the dynamical system
associated to a square integrable stationary $G$-process $\left\{ X_{g}\right\} _{g\in G}$
where $G$ is a countable abelian group. If the spectral measure of
$X_{e}$ is singular, then $\left(X,\mathcal{A},\mu,T^{g}\right)$
has zero Poisson entropy.

In particular, if $\left(X,\mathcal{A},\mu,T^{g}\right)$ has singular
maximal spectral type, then it has zero Poisson entropy.
\end{cor}
For $\mathbb{Z}$-action, we can be more precise.
\begin{prop}
If $\left(X,\mathcal{A},\mu,T\right)$ has totally positive Poisson
entropy, its maximal spectral type is Lebesgue countable.

If $\left(X,\mathcal{A},\mu,T\right)$ has positive Poisson entropy
and possesses a Poisson-Pinsker factor, then the maximal spectral
type on the orthocomplement of the Poisson-Pinsker factor is Lebesgue
countable.\end{prop}
\begin{proof}
The first statement follows from Proposition 10.2 in \cite{Jan08entropy},
combined with Proposition \ref{pro:SpectrumInfinitemeasure}. The
second can also be deduced from an adaptation of Proposition 10.2
in \cite{Jan08entropy} but is also a direct application of Theorem
3.2 in \cite{DanRud07Condent} combined with the fact that conditional
Poisson entropy coincides with conditional Krengel entropy as proved
in \cite{Jan08entropy}.
\end{proof}
Examples of totally positive Poisson entropy transformations are given
by shift associated to null recurrent Markov chains (see \cite{Roy07Infinite}).
\begin{cor}
\label{cor:multiplicity}If $\left(X,\mathcal{A},\mu,T\right)$ has
finite multiplicity, then it has zero Poisson entropy.
\end{cor}
The conclusion of Corollary \ref{cor:multiplicity} is not true if
$\mu$ is not continuous (think of the shift on $\mathbb{Z}$).

\bibliographystyle{plain}
\bibliography{biblio.bib}

\end{document}